\documentclass[11pt]{article}
\usepackage{amsmath,amsthm,amsfonts,amssymb,amscd, amsxtra, mathrsfs,amssymb,amstext,amscd,amsxtra}
\usepackage{url}
\usepackage{pgfplots}
\usepackage{tikz}
\usepackage{mathtools}

\usepackage{booktabs}

\pgfplotsset{width=18em,every axis/.append style={axis lines=center,ylabel style={at={(ticklabel cs:0.5)},rotate=90,anchor=south},,xlabel style={at={(ticklabel cs:0.5)},anchor=north}}}

\usepackage[margin=2.8 cm,nohead]{geometry}
\usepackage[active]{srcltx}
\usepackage{verbatim}
\usepackage{url}
\usepackage{cite}
\usepackage{enumerate}
\usepackage[ruled]{algorithm2e} 
\usepackage{graphicx,float}
\usepackage{mathrsfs,threeparttable,color,soul,colortbl,hhline,rotating,array,multirow,makecell}

\newcommand{\R}{\mathbb{R}}
\newcommand{\N}{\mathbb{N}}
\newcommand{\PJ}{\cal{P}}
\newcommand{\RF}{\cal{R}}
\newcommand{\DR}{\cal{T}}
\newcommand{\id}{\mbox{Id}}
\newcommand{\Fix}{\mbox{Fix~}}
\newtheorem{theorem}{Theorem}
\newtheorem{lemma}[theorem]{Lemma}
\newtheorem{definition}{Definition}
\newtheorem{corollary}[theorem]{Corollary}
\newtheorem{proposition}[theorem]{Proposition}

\newtheorem{remark}{Remark}

\usepackage{epstopdf}
\usepackage{algorithmic}
\usepackage{todonotes}

\usepackage{graphicx,float}
\usepackage{mathrsfs,threeparttable,color,soul,colortbl,hhline,rotating,array,multirow,makecell}
\begin{document}
%%%%%%%%%%%%%%%%%%%%%%%%%%%%%%%%%%%%%%%%%%%%%%%%%%%%%%%%%%%%%
\title{Approximate Douglas-Rachford algorithm  for  two-sets \\convex feasibility problems}
\author{
R. D\'iaz Mill\'an \thanks{School of Information Technology, Deakin University, Melbourne,   Australia, e-mail :{\tt  r.diazmillan@deakin.edu.au}.}
\and
O.  P. Ferreira  \thanks{IME, Universidade Federal de Goi\'as,  Goi\^ania, Brazil, e-mail:{\tt  orizon@ufg.br}. This authors was supported in part by  CNPq grants 305158/2014-7 and 302473/2017-3,  FAPEG/PRONEM- 201710267000532.}
\and
 J. Ugon  \thanks{School of Information Technology, Deakin University, Melbourne,   Australia,  e-mail:{\tt  j.ugon@deakin.edu.au}.}
 }
\maketitle
\begin{abstract}
In this paper, we propose a new algorithm combining the Douglas-Rachford (DR) algorithm and the Frank-Wolfe algorithm, also known as the conditional gradient (CondG) method, for solving the classic convex feasibility problem. Within the algorithm,  which will be named {\it Approximate Douglas-Rachford (ApDR) algorithm}, the CondG method is used as a subroutine to compute feasible inexact projections on the sets under consideration, and the ApDR iteration is defined based on the DR iteration. The ApDR algorithm generates two sequences,  the main sequence, based on the  DR  iteration, and its corresponding shadow sequence. When the intersection of the feasible sets is nonempty, the main sequence converges to a fixed point of the usual DR operator, and the shadow sequence converges to the solution set. We provide some numerical experiments to illustrate the behaviour of the sequences produced by the proposed algorithm.

\end{abstract}
{\bf AMS:} 65K05,   	90C30,  	90C25.
%%%%%%%%%%%%%%%%%%%%%%%%%%%%%%%%
\section{Introduction}
This paper  addresses the classic two-sets convex feasibility problem in finite-dimensional Euclidean space. This problem is formally stated as follows:
\begin{equation}\label{def:InexactProjProb}
\mbox{find } x^*\in A\cap B,
\end{equation}
where the feasible sets   $A, B \subset \mathbb{R}^n$ are convex, closed,  and nonempty sets. A huge variation of practical applications in different areas of mathematics and physical sciences can be modelled in this format, and so Problem (\ref{def:InexactProjProb}) has attracted the attention of many researchers, for example, see \cite{AragonBorweinTam2014, Combettes1996, Combettes1999, HesseLukeNeumann2014} and references therein. Among the various algorithms to solve Problem~\eqref{def:InexactProjProb}, the  {\it Douglas-Rachford algorithm} is one of the most interesting with a long history dating back to the 1950s, see \cite{DouglasRachford1956}.  When applied to the Problem \eqref{def:InexactProjProb}, we can say that this method belongs to the class of the projection algorithm since, at each iteration, a projection on each of the feasible sets is computed, for example, see \cite{artacho2020douglas2019} and the references therein. 

% Because the various practical applications in different areas of mathematics and the physical sciences can be modeled as Problem (1), it has received great attention throughout te years, for example see \cite{AragonBorweinTam2014, Combettes1996, Combettes1999, HesseLukeNeumann2014}.  Among the various algorithm  to solve Problem~\eqref{def:InexactProjProb}, the  {\it Douglas-Rachford algorithm (DR algorithm)} is one of the most interesting  with a long history dating back to the 1950s, see \cite{DouglasRachford1956}. This algorithm, when applied to Problem~\eqref{def:InexactProjProb}, can be seen as belonging to the class of projection algorithms. Indeed, to perform an iteration of the DR algorithm applied to Problem~(1), projections onto feasible sets must be computed first, for example see \cite{artacho2020douglas2019}.

This paper aims to present a new algorithm to solve  Problem~\eqref{def:InexactProjProb}. The proposed algorithm, called {\it Approximate Douglas-Rachford algorithm (ApDR algorithm)},  combines the {\it Douglas-Rachford (DR) algorithm} with the {\it conditional gradient method  (CondG method)} also known as {\it  Frank-Wolfe algorithm}, see  \cite{FrankWolfe1956, LevitinPoljak1966}.  The CondG method is used as the subroutine of the ApDR to compute feasible inexact projections onto the sets under consideration, defining the iteration based on the DR iteration. Specifically, we use the CondG method to minimize the distance between the point to project and the convex set under consideration. The CondG method is a feasible directions method that minimizes the objective function in each iteration, and returns a feasible approximate solution, that is, an inexact projection onto the set in consideration, see \cite{BeckTeboulle2004, Jaggi2013}.  We introduce a relative error criterion to lower the cost of each projection while preserving the convergence of the ApDR algorithm. The ApDR algorithm generates two sequences, one  $(x^k)_{k\in\mathbb{N}}$ based on the DR iteration and the other $(y^k_A)_{k\in\mathbb{N}} \subset A$ corresponding to the shadow sequence of approximate projections. The main results of this paper are as follows. If  $A\cap B\neq \varnothing$, then  $(x^k)_{k\in\mathbb{N}}$  converge to a point  $x^*$, which is a fixed point of the usual  Douglas-Rachford operator,  $(y^k_A)_{k\in\mathbb{N}}$ converges  to ${\PJ}_A(x^*)$  the projection of  $x^*$ onto $A$ and ${\PJ}_A(x^*)\in A\cap B$.

The idea of using the CondG method to compute inexact projections onto the sets $A$ and $B$, instead of the exact ones, is particularly interesting from a computational point of view. It may not be necessary to compute the exact projections when the iterates are far from $A\cap B$, since it adds to the computational cost that the DR algorithm spends computing the projections.  It is noteworthy that the CondG method is easy to implement, has a low computational cost per iteration, and readily exploits separability and sparsity, resulting in high computational performance in different classes of compact sets, see \cite{Dunn1980, FreundGrigas2016, GarberHazan2015, Jaggi2013,  LevitinPoljak1966}.  All these features added to the ApDR algorithm make it very attractive from a computational point of view. In addition, for feasible sets that can only be accessed efficiently through a linear programming oracle, the ApDR algorithm is of considerable interest.  It is worth mentioning that the idea of using the CondG method to compute approximate projections has already been used in several papers. Indeed,  the most direct precursor is \cite{LanZhou2016} which uses the CondG method to compute inexact projections in an accelerated first-order method for solving the constrained optimisation problem. For more applications of the CondG method as a subroutine to compute inexact projections,  we refer the reader to \cite{CardereraPokutta2020, deOliveiraFerreiraSilva2019, ReinierOrizonLeandro2019, MaxJefferson2017, MaxJeffrenato2020}. Another method for solving Problem \eqref{def:InexactProjProb}, which don't use exact projection onto the feasible sets $A$ and $B$, can be found in \cite{DiazMillan2019}. To the best of our knowledge, the combination of the CondG method with the DR algorithm for solving Problem~\eqref{def:InexactProjProb}  has not yet been considered.

The organization of the paper is as follows. In section \ref{sec:Preliminares}, we present some notation and basic results used throughout the paper. In section \ref{Sec:CondG} we describe the   CondG  method  and present some results related to inexact version of  projection, reflection and Douglas-Rachford operator.   In  Sections~\ref{se:DR-Agorithm}  we present ApDR algorithm and its convergence analysis.   Some numerical experiments are provided in Section~\ref{Sec:NumExp}.   We conclude the paper with some remarks in Section~\ref{Sec:Conclusions}.

%%%%%%%%%%%%%%%%%%%%%%%%%%%%%%%%%%%%%%%%%%%%%%%%%%%%%%%%
\section{{Preliminaries}} \label{sec:Preliminares}
%%%%%%%%%%%%%%%%%%%%%%%%%%%%%%%%%%%%%%%%%%%%%%%%%%%%%%%%%%
In this section, we present  some preliminary  results used throughout the paper.  We denote: ${\mathbb{N}}=\{0,1,2,3, \ldots\}$, $\langle \cdot,\cdot \rangle$ is the usual inner product and  $\|\cdot\|$ is the Euclidean norm.  Let    $C\subset  \mathbb{R}^n $ be closed, convex and nonempty  set, the {\it projection and the reflection} are, respectively,   the maps    ${\PJ}_C:   \mathbb{R}^n \to C$ and  ${\RF}_C:   \mathbb{R}^n \to  \mathbb{R}^n$  defined by
\begin{equation} \label{eq:prdef}
{\PJ}_C(v):=  \arg\min_{z \in  C}\|v-z\|, \qquad   {\RF}_C(v):= 2{\PJ}_C(v)-v.
\end{equation} 
In the next lemma we present some important properties of the  projection mapping. 
\begin{lemma}\label{le:projeccion}
 Given a convex and closed set $C\subset \mathbb{R}^n$ and for all $v\in  \mathbb{R}^n$, the following properties hold:
\begin{enumerate}
\item[(i)]$\langle v-{\PJ}_C(v), z-{\PJ}_C(v)\rangle\leq0$, for all $z\in C$;
\item[(ii)] $\|{\PJ}_C(v)-z\|^2\leq \|v-z\|^2- \|{\PJ}_C(v)-v\|^2$, for all $z\in C$;  
\item[(iii)]  the projection mapping ${\PJ}_C$ is continuous.
\end{enumerate}
\end{lemma}
\begin{proof}
The items (i) and (iii)  are proved in  \cite[Proposition 3.10, Theorem 3.14]{BauschkeCombettes20011}. For item (ii),  combine  $ \|v-z\|^2=\|{\PJ}_C(v)-v\|^2+\|{\PJ}_C(v)-z\|^2-2\langle {\PJ}_C(v)-v, {\PJ}_C(v)-z\rangle$ with  item (i). 
\end{proof}
Let  $C,D \subset \mathbb{R}^n$ be  convex, closed,  and nonempty sets.  Let   $T: C \rightrightarrows D $ denotes  a set-valued operator that maps a point in $C$ to a subset of $D$ (i.e. $T(x)\subset  D$ for all $x \in  C$). In the case when $T(x)$ is a singleton for all $x \in C$, $T$ is said to be a single-valued mapping, which is denoted as $T: C \to D $, and write $T (x) = z$ whenever $T (x) = \{z\}$. For a  given  set-valued operator $T: C \rightrightarrows D $ and $y\in C$, we set $T(x)+y:=\{z+y:  z\in T(x) \}$. The set of fixed points of an operator $T$, denoted by ${\Fix}~T$, is defined  by ${\Fix}~T:=\{x\in C: ~x\in T(x)\}.$ The {\it identity operator} is the mapping $\id: \mathbb{R}^n\to \mathbb{R}^n$ that maps every point to itself.
Let us  recall the {\it Douglas-Rachford operator} associated to closed and  convex  sets $A$ and $B$, defined by 
\begin{equation}\label{eq:DR-Operator}
{\DR}_{A,B}(u)=  \frac{1}{2}\left({\RF}_B\circ {\RF}_A+\id \right)(u).
\end{equation}
The next proposition gives a relationship between the fixed point of the Douglas-Rachford operator and the solution set of the feasibility problem, its proof can be found in  \cite{artacho2020douglas2019}. 
\begin{proposition}\label{pr:cfp}
  Let $A,B\subset \mathbb{R}^n$ be  convex closed  sets. Then, $A\cap B \neq \varnothing$ if only if  ${\Fix}{\DR}_{A,B}\neq \varnothing$. Furthermore,  ${\PJ}_A(u) \in  A\cap B \neq \varnothing $ for all $u\in {\Fix}{\DR}_{A,B}$.
\end{proposition}
%%%%%%%%%%%%%%%%%%%%%%%
\section{Conditional gradient (CondG) method} \label{Sec:CondG}
In tis section   we recall the classical {\it conditional gradient method (CondG)}, see for example \cite{BeckTeboulle2004}, and  the concept of feasible inexact projection operator,  see \cite{deOliveiraFerreiraSilva2019}. Associated to the  feasible inexact projection operator, we introduce the concept of inexact reflection operator  and the inexact Douglas-Rachford  operator. We also present some    important  properties of these operators.  

For presenting  CondG method, we assume the existence of a linear optimisation oracle (or simply LO oracle) capable of minimising linear functions over the constraint set $C$.  We formally state the  CondG method  in  Algorithm~\ref{Alg:CondG}  to calculate an inexact projection of the point  $u\in {\mathbb R}^n$  onto a compact convex set $C$  relative   to ${y}\in C$ and  ${v}\in {\mathbb R}^n$ with    {\it forcing parameter}  $\epsilon \geq 0$.

\begin{algorithm}[ht] \label{Alg:CondG}
	\caption{{\bf CondG$_{C}$ method} ${y^+_C}\in  {\PJ}^{\epsilon}_C({y},{v}, u)$}
	\begin{algorithmic}[1]
	\STATE {Take $\epsilon>0$,  ${y}\in C$ and  ${v}, u\in {\mathbb R}^n$. Set $w_0={y}$ and  $\ell=0$.}
	\STATE{ Use a LO oracle to compute an optimal solution $z_\ell$ and the optimal value $s_{\ell}^*$ as
\begin{equation}\label{eq:CondG_{C}}
z_\ell := \arg\min_{z \in  C} \,\langle w_\ell-u, ~z-w_\ell\rangle,  \qquad s_{\ell}^*:=\langle  w_\ell-u, ~z_\ell-w_\ell \rangle.
\end{equation}}
	\STATE{If $-s^*_{\ell}\leq  \epsilon \|{y}-{v}\|^2$, then {\bf stop}, and  set ${{y^+_C}}:=w_\ell$. Otherwise,  set 
\begin{equation}\label{eq:stepsize}
w_{\ell+1}:=w_\ell+ \alpha_\ell(z_\ell-w_\ell), \qquad {\alpha}_\ell: =\min\Big\{1, \frac{-s^*_{\ell}}{\|z_\ell-w_\ell\|^2}  \Big\}.
\end{equation}}
        \STATE{ Set $\ell\gets \ell+1$, and go to Step 2.}
	\end{algorithmic}
\end{algorithm}
Let us  describe the main features of  CondG method; for further details,  see, for example,  \cite{BeckTeboulle2004, Jaggi2013, LanZhou2016}. Let  $u\in {\mathbb R}^n$,    $\psi_u: \mathbb{R}^n \to \mathbb{R}$ be defined by $\psi_u(z):= \|z -u\|^2/2$,  and  $C\subset  {\mathbb R}^n$ a convex compact set. The CondG method is a specialized version of the classical conditional gradient method  applied to the problem $\min_{z \in C}\psi_u(z)$. In this case,  \eqref{eq:CondG_{C}} is equivalent to $s_{\ell}^*:=\min_{z \in C}\langle \psi_u'(w_\ell) ,~z-w_\ell\rangle$. Since the function $\psi_u$ is convex we have $\psi_u(z)\geq \psi_u(w_\ell) + \langle \psi_u'(w_\ell) ,~z-w_\ell\rangle\geq    \psi_u(w_\ell)  +   s_{\ell}^*$, for all $z\in C$. Set  $ w_*:=\arg \min_{z \in C}\psi_u(z)$ and  $\psi^*:= \psi_u(w_*)$. Letting $z= w_*$ in the last inequality  we have $\psi_u(w_\ell)\geq \psi^* \geq \psi_u(w_\ell)  +   s_{\ell}^*$, which implies that $s_{\ell}^*\leq 0$. Thus, $
-s_{\ell}^*=\langle  u-w_\ell, ~z_\ell-w_\ell \rangle \geq 0\geq  \langle  u-w_*, ~z-w_* \rangle$,  for all $z\in C$. 
Therefore, we can set the  stopping criterion to $-s_{\ell}^*\leq \epsilon \|{y}-{v}\|^2$, which ensures that CondG will terminate after finitely many iterations. When the  CondG  method computes  $w_\ell \in C$ satisfying $-s_{\ell}^*\leq  \epsilon \|{y}-{v}\|^2 $, then the method terminates. Otherwise, it computes the stepsize $\alpha_\ell = \arg\min_{\alpha \in [0,1]} \psi_u(w_\ell + \alpha(z_\ell - w_\ell))$  using exact line search.  Since $z_\ell$, $w_\ell \in C$  and $C$ is convex, we conclude from  \eqref{eq:stepsize}  that $w_{\ell+1} \in C$, thus the   CondG method generates a sequence in $C$.  Finally,   \eqref{eq:CondG_{C}} implies that  $\langle  u-w_\ell, ~z-w_\ell\rangle\leq -s_{\ell}^*$, for all  $ z\in C$.  Hence, considering  the  stopping criterion      $-s_{\ell}^*\leq \epsilon \|{y}-{v}\|^2$, we conclude that {\it any  output of CondG method ${{y^+_C}}\in C$ is  a feasible inexact projection onto $C$} 
 of the point $u\in {\mathbb R}^n$ with respect to ${y}\in C$, ${v}\in {\mathbb R}^n$ and relative  error tolerance function   $\epsilon \|{y}-{v}\|^2$, i.e.,   
\begin{equation}\label{eq:IIPJ}
\langle  u-{{y^+_C}}, ~z-{{y^+_C}}\rangle\leq  \epsilon \|{y}-{v}\|^2 \qquad \forall~z\in C.
\end{equation}
Inspired by \eqref{eq:IIPJ}, in the following  we present  the  feasible inexact projection operator associated to Algorithm~\ref{Alg:CondG}, see \cite{deOliveiraFerreiraSilva2019}. 
\begin{definition} \label{def:InexactProjC}
 The {\it feasible inexact projection operator onto $C$} relative to ${y}\in C$ and  ${v}\in {\mathbb R}^n$ with   forcing parameter $\epsilon\geq 0$, denoted by ${\PJ}^{\epsilon}_C({y},{v}, \cdot): {\mathbb R}^n \rightrightarrows C$,  is defined as follows
	\begin{equation} \label{eq:Projw} 
		{\PJ}^{\epsilon}_C({y},{v}, u):= \left\{{y^+}\in C:~\langle  v-{y^+}, ~z-{y^+}\rangle\leq  \epsilon \|{y}-{v}\|^2 , ~ \forall~z\in C \right\}.
	\end{equation}
	Each point ${{y^+_C}}\in {\PJ}^{\epsilon}_C({x},{y}, u)$ is called a feasible inexact projection of $u$ onto $C$ relative to ${x}\in C$ and  ${y}\in {\mathbb R}^n$  with   forcing parameter $\epsilon\geq 0$.
\end{definition}
\begin{remark} \label{re:eipp}
Let $C$ be a  closed and convex set,   $v\in {\mathbb R}^n $,  ${y}\in C$,   ${v}\in {\mathbb R}^n$  and  $\epsilon\geq 0$. It follows from Definition~\ref{def:InexactProjC}, item $(i)$ of Lemma~\ref{le:projeccion} and   $\epsilon \geq 0$ that, ${\PJ}_C(u)\in{\PJ}^{\epsilon}_C({y},{v}, u)$. Hence,  ${\PJ}^{\epsilon}_C({y},{v}, u) \neq \varnothing$. For $\epsilon= 0$,   Lemma~\ref{le:projeccion} (i)  together with \eqref{eq:IIPJ}  implies  that  ${\PJ}_C(u)={\PJ}^{0}_C({y},{v}, u)$,  for all $u\in {\mathbb R}^n$, ${y}\in C$ and   ${v}\in {\mathbb R}^n$. It is worth to noting that  in order to  use    Algotithm~\ref{Alg:CondG} for projecting exact and inexact   onto  $C$ we need to assume that the   LO oracle can be used  effectively. In particular, we need to assume that the set  $C \subset \mathbb{R}^n$ is  compact. For an example where  LO oracle cannot be used  effectively, see  \cite{Rothvoss2017}.
\end{remark}
Below we present a particular counterpart of the  firm  non-expansiveness of the projection operator  to  feasible inexact projection operator.
\begin{proposition} \label{pr:snonexp}
  Let $u, v\in {\mathbb R}^n$,  ${y_C} \in C$  and   $ \epsilon\geq 0$. If     ${y^+_C}\in {\PJ}^{\epsilon}_C({y_C}, {v},u)$ and  ${\bar x_C}= {\PJ}_C({\bar w})$, then  
 \begin{equation*}
 \|{y^+_C}-{\bar x_C}\|^2\leq \|u-{\bar w}\|^2- \|(u-{y^+_C})-({\bar w}-{\bar x_C})\|^2  +2\epsilon\|{y_C}-{v}\|^2.
\end{equation*} 
\end{proposition}
\begin{proof}
Since   ${y^+_C}\in {\PJ}^{\epsilon}_C({y_C}, {v},u)$ and  ${\bar x_C}= {\PJ}_C({\bar w})$,  it follows from \eqref{eq:Projw} and Lemma~\ref{le:projeccion} that
$$
 \langle  {y^+_C}-u, ~{y^+_C}-{\bar x_C}\rangle\leq  \epsilon\|{y_C}-{v}\|^2,  \qquad \qquad  \langle  {\bar w}-{\bar x_C}, ~{y^+_C}-{\bar x_C}\rangle\leq 0.
$$ 
By adding the last two inequalities,   some algebraic manipulations yield 
\begin{equation*}
-\langle  u-{\bar w}, ~{y^+_C}-{\bar x_C}\rangle \leq  -\|{y^+_C}-{\bar x_C}\|^2 + \epsilon\|{y_C}-{v}\|^2.
\end{equation*}
Since  $\|(u-{y^+_C})-({\bar w}-{\bar x_C})\|^2=\|u-{\bar w}\|^2-2\langle u-{\bar w}, {y^+_C}-{\bar x_C}\rangle +\|{y^+_C}-{\bar x_C}\|^2$, the desired inequality follows by combination with the last inequality.
\end{proof}
\begin{proposition}\label{prop:dist}
Let $u, v\in {\mathbb R}^n$,  ${y_C} \in C$ and   $ \epsilon\geq 0$. If     ${y^+_C}\in {\PJ}^{\epsilon}_C({y_C}, {v},u)$, then   $\|{y^+_C}-P_C(u)\|\leq \sqrt{2\epsilon}\|{y_C}- {v}\|$.  
\end{proposition}
\begin{proof}
Since  ${y^+_C}\in {\PJ}^{\epsilon}_C({y_C}, {v},u)$ , using Definition~\ref{def:InexactProjC}   we have   $\langle u-{y^+_C},z-{y^+_C}\rangle\leq \epsilon\|{y_C}- {v}\|^2$,  for all $z\in C$. Thus, some algebraic manipulations give $\|{y^+_C}-z\|^2+\langle u-z,z-{y^+_C} \rangle\leq  \epsilon\|{y_C}- {v}\|^2$, for all 
 $z\in C$. Since  $\|u-{y^+_C}\|^2=\|u-z\|^2+2\langle u-z,z-{y^+_C}\rangle +\|z-{y^+_C}\|^2$,  for all $z\in C$,  we have  
\begin{equation*}
\|{y^+_C}-z\|^2+\frac{1}{2}(\|{y^+_C}-u\|^2-\|u-z\|^2-\|z-{y^+_C}\|^2)\leq  \epsilon\|{y_C}- {v}\|^2, \qquad  \forall z\in C, 
\end{equation*}
which is equivalent to $\|{y^+_C}-z\|^2+\|{y^+_C}-u\|^2-\|z-u\|^2\leq 2\epsilon\|{y_C}- {v}\|^2$,  for all $z\in C$.
Substituting $z=P_A(u)$  into the last inequality  and using $\|P_A(u)-u\|\leq \|{y^+_C}-u\|$ the desired result follows.
\end{proof}
Next we introduce the concept of  inexact reflection operator  associated to a convex set  $C$.
\begin{definition} \label{def:InexactRop}
  The inexact reflection operator  associated to $C$ relative to ${y_C}\in C$ and  ${v}\in {\mathbb R}^n$  with   forcing parameter $\epsilon\geq 0$,  denoted by ${\RF}^{\epsilon}_C({y_C}, v,\cdot): {\mathbb R}^n \rightrightarrows {\mathbb R}^n$,  is a set-valued mapping defined as follows ${\RF}^{\epsilon}_C({y_C}, v,u):=\{ 2{y^+_C}-u:~ {y^+_C}\in {\PJ}^{\epsilon}_C({y_C}, {v},u)\}$, or equivalently, 
\begin{equation} \label{eq:IRF}
{\RF}^{\epsilon}_C({y_C}, v,u):= 2{\PJ}^{\epsilon}_C({y_C}, {v},u)-u.
\end{equation}
	Each point belonging to  ${\RF}^{\epsilon}_C({y_C}, v,u)$ is called a inexact reflection of ${u}$ with respect  $C$ relative to ${y_C} \in C$ and  ${v}\in {\mathbb R}^n$  with   forcing parameter $\epsilon\geq 0$.
\end{definition}
In the following  we present a particular  counterpart to the  non-expansiveness of the reflection  operator   to inexact reflection operator.
\begin{proposition} \label{pr:nonexpr}
  Let $u, v\in {\mathbb R}^n$,  ${y_C}, {\bar x_C}\in C$,   ${v}\in {\mathbb R}^n$,  $ \epsilon\geq 0$ and  $\delta \geq 0$. If     ${y^+_C}\in {\PJ}^{\epsilon}_C({y_C}, {v},u)$ and  ${\bar x_C}= {\PJ}_C({\bar w})$, then  $2{y^+_C}-u\in {\RF}^{\epsilon}_C({y_C}, v,u)$,  $2{\bar x_C}-{\bar w}\in {\RF}({\bar w})$, and 
 \begin{equation*}
\|(2{y^+_C}-u)-(2{x_C}-{\bar w})\|^2\leq  \|u-{\bar w}\|^2+  4\epsilon\|{y_C}- {v}\|^2.
\end{equation*} 
\end{proposition}
\begin{proof}
The first inclusion follows from Definition~\ref{def:InexactRop} and the second from \eqref{eq:prdef}.
 To prove the inequality, first note that  direct computation yields 
$$
\|(2{y^+_C}-u)-(2{\bar x_C}-{\bar w})\|^2=\|u-{\bar w}\|^2+2\left(\|{y^+_C}-{\bar x_C}\|^2-\|u-{\bar w}\|^2+  \|(u-{y^+_C})-({\bar w}-{\bar x_C})\|^2\right).
$$
Since     ${y^+_C}\in {\PJ}^{\epsilon}_C({y_C}, {v},u)$ and  ${\bar x_C}= {\PJ}_C({\bar w})$,  the desired inequality follows from the last equality together with   Proposition~\ref{pr:snonexp}. 
\end{proof}

In order to introduce  the approximate Douglas-Rachford algorithm in the next section,  it is necessary to introduce first the  approximate Douglas-Rachford (ApDR) operator.  
\begin{definition} \label{def:InexactDRop}
Let  $A, B \subset \mathbb{R}^n$ be  closed convex sets.  The  inexact Douglas-Rachford (IDR) operator associated to the sets $A$ and $B$ relative to ${y_A} \in A$ and    ${y_B} \in B$   with   forcing parameters  $\epsilon\geq 0$ and $\delta \geq 0$,  denoted by ${\DR}^{\epsilon, \delta}_{A,B}({y_A},{y_B}, \cdot): {\mathbb R}^n \rightrightarrows {\mathbb R}^n$,  is a set-valued mapping defined as follows 
\begin{equation} \label{eq:IDR1}
{\DR}^{\epsilon, \delta}_{A,B}({y_A},{y_B}, u):=\big\{ u+{y_B^+}-{y_A^+}:~{y_A^+}\in  {\PJ}^{\epsilon}_A({y_A},{y_B}, u), ~{y_B^+}\in {\PJ}^{\delta}_B({y_B},{y_A},2{y_A^+}-u)\big\}.
\end{equation}
or equivalently, 
\begin{equation} \label{eq:IDR2}
{\DR}^{\epsilon, \delta}_{A,B}({y_A},{y_B}, u):= \frac{1}{2}\left( {\RF}^{\delta}_B\big({y_B},{y_A},{\RF}^{\epsilon}_A({y_A},{y_B}, u)\big)+u\right).
\end{equation}
\end{definition}
\begin{remark}  \label{re:pctab}
 For $\epsilon=0$,  it follows from Definition~\ref{def:InexactProjC}  that  ${\PJ}_C(u)={\PJ}^{0}_C({y_A},{y_B},u)$, for all $u\in {\mathbb R}^n$, for all ${y_A} \in C$ and    ${y_B} \in {\mathbb R}^n$. Thus, the second equality in \eqref{eq:prdef} and  \eqref{eq:IRF} imply that that  ${\RF}_C(u)={\RF}^{0}_C({y_A},{y_B},  u)$, for all $u\in {\mathbb R}^n $, ${y_A} \in C$ and    ${y_B} \in {\mathbb R}^n$.   Furthermore, it follows from \eqref{eq:DR-Operator} and \eqref{eq:IDR2} that ${\DR}_{A,B}(u)={\DR}^{0, 0}_{A,B}({y_A},{y_B},u)$,  for all $u\in {\mathbb R}^n $, ${y_A}\in A$ and ${y_B}\in B$.
\end{remark}
We end this section with an important result that we need to  prove our main result. 
\begin{proposition}\label{pr:frfpfr}
Let $A,B\subset \mathbb{R}^n$ be  convex closed sets such that $A\cap B \neq \varnothing$, ${y}_A \in A$ and    ${y}_B \in B$,  $\epsilon\geq 0$ and $\delta\geq 0$.  Take  $x \in \mathbb{R}^n$, ${y_A^+}\in  {\PJ}^{\epsilon}_A({y}_A,{y}_B, x)$, ${y_B^+}\in {\PJ}^{\delta}_B({y}_B,{y}_A,2{y_A^+}-x)$, and set  $x^{+}:=x+{y_B^+}-{y_A^+}$.  If  ${\bar x}\in {\Fix}{\DR}_{A,B}$, then 
 \begin{equation*}
 \|x^{+}-{\bar x}\|^2\leq \|x-{\bar x}\|^2-\|x-x^+\|^2 +2(\epsilon+\delta)\|{y}_A-{y}_B\|^2.
 \end{equation*} 
\end{proposition}
\begin{proof}
Since ${\bar x}\in {\Fix}{\DR}_{A,B}$ we have ${\PJ}_A({\bar x})= {\PJ}_B(2 {\PJ}_A({\bar x})-{\bar x})$. Set ${\bar x}_A=  {\PJ}_A({\bar x})$ and   ${\bar x}_B= {\PJ}_B(2{\bar x}_A-{\bar x})$. In this case, ${\bar x}_A={\bar x}_B$.   Due to  ${y_B^+}\in {\PJ}^{\delta}_B({y}_B,{y}_A,2{y_A^+}-x)$ and  ${\bar x}_B= {\PJ}_B(2{\bar x}_A-{\bar x})$,  Proposition~\ref{pr:nonexpr}  with  $C=B$,  $\epsilon=\delta$, $u=2{y_A^+}-x$,  $v={y_A}$, ${\bar x_C}={\bar x}_B$ and ${\bar w}=2{\bar x}_A-{\bar x}$ yields   
 \begin{equation*}
\|(2{y^+_B}-(2{y_A^+}-x))-(2{\bar x_B}-2({\bar x}_A-{\bar x})\|^2\leq   \|(2{y_A^+}-x)-(2{\bar x}_A-{\bar x})\|^2+  4\delta\|{y}_B-{y}_A\|^2.
\end{equation*} 
 Thus, taking into account that    ${\bar x}_A={\bar x}_B$ we we conclude that 
 \begin{equation*}
\|(2{y^+_B}-(2{y_A^+}-x))-{\bar x}\|^2\leq  \|(2{y_A^+}-x)-(2{\bar x}_A-{\bar x})\|^2+  4\delta\|{y}_B-{y}_A\|^2.
\end{equation*} 
Considering that ${y_A^+}\in  {\PJ}^{\epsilon}_A({y}_A,{y}_B, x)$ and ${\bar x}_A=  {\PJ}_A({\bar x})$  we  apply again    Proposition~\ref{pr:nonexpr} with  $C=A$, $u=x$,  $v={y_B}$, ${\bar x_C}={\bar x}_A$ and ${\bar w}={\bar x}$  to obtain 
 \begin{equation*}
\|(2{y^+_A}-x)-(2{x_A}-{\bar x})\|^2\leq  \|x-{\bar x}\|^2+  4\epsilon\|{y_A}- {y_B}\|^2.
\end{equation*} 
Hence, combining two previous inequality we conclude that 
 \begin{equation} \label{eq:ieqtp}
\|(2{y^+_B}-(2{y_A^+}-x))-{\bar x}\|^2\leq  \|x-{\bar x}\|^2+  4(\epsilon+\delta)\|{y}_B-{y}_A\|^2.
\end{equation} 
On the other  hand,   $x^{+}-{\bar x}=\frac{1}{2}(x-{\bar x})+\frac{1}{2}\left(2{y_B^+}-(2{y_A^+}-x)-{\bar x}]\right).$
Thus, direct computation shows that
\begin{equation} \label{eq:cfps1}
 \|x^{+}-{\bar x}\|^2=\frac{1}{4}\|x-{\bar x}\|^2+\frac{1}{2}\langle x-{\bar x}, 2{y_B^+}-(2{y_A^+}-x)-{\bar x}]\rangle +  \frac{1}{4}\|2{y_B^+}-(2{y_A^+}-x)-{\bar x}]\|^2.
\end{equation} 
Since $ x-x^{+}=\frac{1}{2}(x-{\bar x})-\frac{1}{2}\left([2{y_B^+}-(2{y_A^+}-x)]-{\bar x}]\right)$, we have 
\begin{equation} \label{eq:cfps2}
 \|x-x^{+}\|^2=\frac{1}{4}\|x-{\bar x}\|^2-\frac{1}{2}\langle x-{\bar x}, 2{y_B^+}-(2{y_A^+}-x)-{\bar x}]\rangle +  \frac{1}{4}\|2{y_B^+}-(2{y_A^+}-x)-{\bar x}]\|^2.
\end{equation}
Therefore, summing equalities \eqref{eq:cfps1} and \eqref{eq:cfps2} we conclude that 
\begin{equation*} 
 \|x^{+}-{\bar x}\|^2+\|x-x^{+}\|^2=  \frac{1}{2}\|x-{\bar x}\|^2+\frac{1}{2}\|2{y_B^+}-(2{y_A^+}-x)-{\bar x}]\|^2, 
\end{equation*}
Therefore,  using   \eqref{eq:ieqtp} we have  $ \|x^{+}-{\bar x}\|^2+\|x-x^{+}\|^2=  \|x-{\bar x}\|^2+  2(\epsilon+\delta)\|{y}_B-{y}_A\|^2$, 
which is equivalent to the desired inequality, and the proof is concluded. 
\end{proof}
%%%%%%%%%%%%%%%%%%%%%%%%%%%%%%%%%%%%%%%%%%%
\section{Approximate Douglas-Rachford algorithm} \label{se:DR-Agorithm}
%%%%%%%%%%%%%%%%%%%%%%%%%%%%%%%%%%%%%%%%%%
The aim of this section is to present  a new algorithm for solving the classic feasibility problem, which we name  Approximate Douglas-Rachford algorithm.   Before presenting the approximate  Douglas-Rachford algorithm, let us first recall the classical Douglas-Rachford algorithm.   At each step of the classic Douglas--Rachford algorithm the previous iterate is first \emph{reflected} through $A$, then \emph{reflected} through  $B$, and finally the resulting point is averaged with the previous iterate. In this case the iterative sequence $\{x^k\}_{k \in \mathbb{N}}$ is defined as 
\begin{equation}\label{eq:DR-step}
x^{k+1} :=    \frac{1}{2}\left( {\RF}_B\big( {\RF}_A(x^k)\big)+x^k\right) =  \frac{1}{2}\left({\RF}_B\circ {\RF}_A+\id \right)(x^k), \qquad  \forall k \in \mathbb{N}.
\end{equation}
Then conceptual {\it approximate  Douglas-Rachford  (ApDR) algorithm}  is stated in Algorithm~\ref{Alg:IDRA}. 

\begin{algorithm}[ht] \label{Alg:IDRA}
	\caption{Approximate Douglas-Rachford  (ApDR) algorithm}
	\begin{algorithmic}[1]
	\STATE {Let   $(\epsilon_k)_{k\in\mathbb{N}}$ and  $(\delta_k)_{k\in\mathbb{N}}$   be  sequences of non-negative real numbers,  $y^0_A\in A$, $y^0_B\in B$ and $x^1\in \R^n$. Set $k=1$.}
	\STATE{  Compute  
\begin{equation} \label{eq:cpAB}
y^k_A \in   {\PJ}^{\epsilon_k}_A\big(y^{k-1}_A,y^{k-1}_B, x^k\big), \qquad y^k_B \in   {\PJ}^{\delta_k}_B\big(y^{k-1}_B,y^{k-1}_A, 2y^k_A-x^k\big).
\end{equation}
 If   $y_B^k=y_A^k$, then {\bf stop}.  Otherwise,  set the next iterate  $x^{k+1}$ as follows  
\begin{equation} \label{eq2:xk+1)}
 x^{k+1}=x^k+y_B^k-y_A^k. 
\end{equation} }
	\STATE{Set $k\gets k+1$, and go to Step~2.}
	\end{algorithmic}
\end{algorithm}
Let us  describe the main features of  Algorithm \ref{Alg:IDRA}. In {Step 1} to compute  the next iterate  $x^{k+1}$ in \eqref{eq2:xk+1)} we need first to compute $y^k_A $ and $y^k_B$ satisfying\eqref{eq:cpAB}. Then,   it follows from  \eqref{eq:IDR1} that  the next iterate $ x^{k+1}$ satisfies 
\begin{equation} \label{eq:Next}
 x^{k+1}\in  {\DR}^{\epsilon_k, \delta_k}_{A,B}\big(y_A^{k-1}, y_B^{k-1},x^k\big).
\end{equation}
If Algorithm \ref{Alg:IDRA}  {stops},  then  $y_B^k=y_A^k$, and   \eqref{eq2:xk+1)} implies that $x^{k+1}=x^{k}$. Hence, it follows from  \eqref{eq:Next} that  $ x^{k}\in  {\DR}^{\epsilon_k, \delta_k}_{A,B}\big(y_A^{k-1}, y_B^{k-1},x^k\big)$.  Therefore, Proposition~\ref{pr:cfp} implies that $y_A^k=y_B^k\in A\cap B$ and we have a solution.
\begin{remark}
If $\epsilon_k= 0$ and $\delta_k= 0$ then  Remark~\ref{re:eipp}  and \eqref{eq:cpAB} imply that  $y^k_A =  {\PJ}_A(x^k)$  and $y^k_B =  {\PJ}_B(2y^k_A-x^k)$. Then, it follows from  \eqref{eq2:xk+1)}  that 
$$
x^{k+1}= x^k+  {\PJ}_B(2 {\PJ}_A(x^k)-x^k)- {\PJ}_A(x^k).
$$
We proceed to show that last equality is equivalent to \eqref{eq:DR-step}. First note that some algebraic manipulations show that 
$$
x^k+  {\PJ}_B(2 {\PJ}_A(x^k)-x^k)- {\PJ}_A(x^k)= \frac{1}{2}\left( 2 {\PJ}_B(2 {\PJ}_A(x^k)-x^k)-(2 {\PJ}_A(x^k)-x^k)+x^k\right).
$$
Hence, using the last equality and the definition of the reflection ${\RF}_A(x^k)$  we conclude that 
$$
x^{k+1}  =  \frac{1}{2}\left( 2 {\PJ}_B({\RF}_A(x^k))-{\RF}_A(x^k) +x^k\right), 
$$
and  by using  the definition of  ${\RF}_B(x^k)$, the last inequality   becomes \eqref{eq:DR-step}.  Therefore, if $\epsilon_k= 0$ and $\delta_k= 0$  then   Conditional  Douglas-Rachford algorithm retrieve the classical Douglas-Rachford algorithm. 
\end{remark}
\begin{remark}
  It follows from Definition~\ref{def:InexactProjC}, the item $(i)$ of Lemma~\ref{le:projeccion},   $\epsilon_k \geq 0$ and $\delta_k \geq 0$ that ${\PJ}_A(x^k)\in    {\PJ}^{\epsilon_k}_A\big(y^{k-1}_A,y^{k-1}_B, x^k\big)$ and ${\PJ}_B(2y^k_A-x^k) \in  {\PJ}^{\delta_k}_B\big(y^{k-1}_B,y^{k-1}_A, 2y^k_A-x^k\big)$, i.e.,  the inexact projection also  accepts an exact one. Therefore,  if the exact projection onto $A$ or onto $B$  is easy to compute then we do not need to use  Algorithm \ref{Alg:CondG} to do that task in obtaining \eqref{eq:cpAB} in the respective set. For instance, the exact projections onto a hyperplane,  box constraint or Lorentz cone is very easy to obtain.  For example,  see  \cite[p. 520]{NocedalWright2006} and \cite[Proposition~3.3]{FukushimaTseng2002}.
\end{remark}
In the following we  present our main result related the convergence of  Algorithm~\ref{Alg:IDRA}.
\begin{theorem}\label{th:conver}
Let $\{x^k\}_{k \in \mathbb{N}^*}$ be  a sequence generated by Algorithm~\ref{Alg:IDRA}. Assume that $A\cap B\neq \emptyset$,  $A \subset \mathbb{R}^n$ is  compact set, $0\leq 2(\epsilon_k+\delta_k)\leq \bar{\epsilon}<1$ and $\bar{\epsilon}\neq 0$.  Then,  $\{x^k\}_{k \in \mathbb{N}^*}$ converges to a point  $x^*\in {\Fix}{\DR}_{A,B}$. As a consequence, $\{y_A^k\}_{k\in\N}$ converges to  $P_A(x^*)\in A\cap B$.
\end{theorem}
\begin{proof}
Let ${\bar x}\in  {\Fix}{\DR}_{A,B}$.  Applying     Proposition~\ref{pr:frfpfr}    with ${y_A}=y^{k-1}_A$, ${y_B}=y^{k-1}_B$, $\epsilon=\epsilon_k$, $x={x^k}$,  ${y_A^+}=y^{k-1}_A$, ${y_B^+}=y^{k-1}_B$, ${x^+}={x^{k+1}}$ and taking into account \eqref{eq2:xk+1)}  we obtain  that  
\begin{equation*} 
\|x^{k+1}-{\bar x}\|^2  \leq  \|x^{k}-{\bar x}\|^2  -\|y_A^k-y_B^k\|^2+2(\epsilon_k+\delta_k)\|y_A^{k-1}-y_B^{k-1}\|^2, \qquad \forall k \in \mathbb{N}.
\end{equation*}
Since $0\leq 2(\epsilon_k+\delta_k)\leq \bar{\epsilon}<1$ we have  $\|x^{k+1}-{\bar x}\|^2  \leq  \|x^{k}-{\bar x}\|^2  -\|y_A^k-y_B^k\|^2+\bar{\epsilon}\|y_A^{k-1}-y_B^{k-1}\|^2$, for all $k \in \mathbb{N}^*$, which is equivalent to 
\begin{equation} \label{eq:fejer}
\|x^{k+1}-{\bar x}\|^2 + \bar{\epsilon}\|y_A^k-y_B^k\|^2 \leq  \|x^{k}-{\bar x}\|^2 +\bar{\epsilon}\|y_A^{k-1}-y_B^{k-1}\|^2 - (1-\bar{\epsilon})\|y_A^k-y_B^k\|^2, \qquad \forall k \in \mathbb{N}^*.
\end{equation}
It follows from \eqref{eq:fejer}  that, for any ${\bar x}\in  {\Fix}{\DR}_{A,B}$,   the sequence $(\|x^{k+1}-{\bar x}\|^2 + \bar{\epsilon}\|y_A^k-y_B^k\|^2 )_{k \in \mathbb{N}^*}$ is monotonous, non-increasing  and  bounded from below by zero, which implies it must converge. Thus, taking into account that the  inequality in  \eqref{eq:fejer} also implies that 
\begin{equation*} 
(1-\bar{\epsilon})\|y_A^k-y_B^k\|^2    \leq  (\|x^{k}-{\bar x}\|^2 +\bar{\epsilon}\|y_A^{k-1}-y_B^{k-1}\|^2)-(\|x^{k+1}-{\bar x}\|^2 + \bar{\epsilon}\|y_A^k-y_B^k\|^2), \qquad \forall k \in \mathbb{N}^*, 
\end{equation*}
and   $0< \bar{\epsilon}<1$, we have   $\lim_{k\to + \infty}\|y_A^k-y_B^k\|=0$. Due to  $(\|x^{k+1}-{\bar x}\|^2 + \bar{\epsilon}\|y_A^k-y_B^k\|^2 )_{k \in \mathbb{N}^*}$  being  convergent, we also conclude that  $\{x^k\}_{k \in \mathbb{N}^*}$ is bounded.  Let $\{x^{k_j}\}_{j \in \mathbb{N}}$ be a converging subsequence of $\{x^k\}$, and ${x^*} = \lim_{j\to +\infty} x^{k_j}$.   We prove that ${x^*}\in {\Fix}{\DR}_{A,B}$. For that, we first note that from the fact that $A$ is a compact set, $\{y_A^k\}_{k \in \mathbb{N}}\subset A$ and  $\lim_{k\to + \infty}\|y_A^k-y_B^k\|=0$,  we conclude that both sequences $\{y_A^k\}_{k \in \mathbb{N}}$ and $\{y_B^k\}_{k \in \mathbb{N}}$ are also bounded. Let ${\bar y}$ be an cluster point of $\{y_A^{k_j}\}_{j \in \mathbb{N}}$ and $\{y_A^{k_\ell}\}_{l \in \mathbb{N}}$ be a subsequence such that  $\lim_{l\to +\infty} y_A^{k_\ell}= {\bar y}$.  Since   $\lim_{k\to + \infty}\|y_A^k-y_B^k\|=0$ we obtain that  $\lim_{l\to + \infty}\|y_B^{k_\ell}-{\bar y}\|\leq \lim_{l\to + \infty}( \|y_A^{k_\ell}-y_B^{k_\ell}\|+ \|y_A^{k_\ell}-{\bar y}\|)=0$. Thus,  $\lim_{l\to +\infty} y_B^{k_\ell}= {\bar y}$.  Hence,  both sequence $\{y_A^{k_\ell}\}_{l \in \mathbb{N}}$ and $\{y_B^{k_\ell}\}_{l \in \mathbb{N}}$ converge to a same point  ${\bar  y}\in A\cap B$. Furthermore,  \eqref{eq:cpAB} and Definition~\ref{def:InexactProjC}  imply  that 
\begin{align*}
\langle  x^{k_\ell}-y_A^{k_\ell}, ~z- y_A^{k_\ell}\rangle & \leq  \epsilon_{k_\ell}\|y_A^{k_\ell-1}-y_B^{k_\ell-1}\|^2, \qquad \forall z\in A,  \\
\langle  (2y^{k_\ell}_A-x^{k_\ell})-y_B^{k_\ell}, ~z- y_B^{k_\ell}\rangle & \leq  \delta_{k_\ell}\|y_A^{k_\ell-1}-y_B^{k_\ell-1}\|^2,  \qquad \forall z\in B.
\end{align*}
Taking the limit as $\ell$ goes to $+\infty$ in these inequalities, we conclude that $ \langle  {x^*}-{\bar y}, ~z- {\bar  y}\rangle\leq  0$,  for all  $z\in A$,  and  $\langle  (2{\bar  y}-{x^*})-{\bar  y}, ~z- {\bar  y}\rangle\leq  0$ for all $z\in B$, 
Hence,  ${\bar y}={\PJ}_A({x^*})$ and  ${\bar  y}={\PJ}_B(2{\bar  y}- {x^*})$ or equivalently,  ${\PJ}_A({x^*})= {\PJ}_B(2{\PJ}_A({x^*})- {x^*})$, which after some algebraic manipulations  yields
$$
{x^*}  =  \frac{1}{2}\left( 2 {\PJ}_B\left({\RF}_A({x^*})\right)-{\RF}_A({x^*}) +{x^*}\right).
$$
Therefore, using  \eqref{eq:DR-Operator} we have  ${x^*}  = {\DR}_{A,B}({x^*})$.  Moreover, $(\|x^{k+1}-{x^*}\|^2 + \bar{\epsilon}\|y_A^k-y_B^k\|^2 )_{k \in \mathbb{N}^*}$  is also  monotonous and  non-increasing. Considering that $\lim_{j\to +\infty} \|y_A^{k_j}-y_B^{k_j}\|=0$,  \eqref{eq2:xk+1)} this implies that $\lim_{j\to +\infty} x^{k_j+1}= {x^*}$, i.e., $\lim_{j\to +\infty} \|x^{k_j+1}-{x^*}\|=0$.  Thus, $\lim_{j\to +\infty}(\|x^{k_j+1}-{x^*}\|^2 + \bar{\epsilon}\|y_A^{k_j}-y_B^{k_j}\|^2 )=0$. Hence, the sequence    $(\|x^{k+1}-{x^*}\|^2 + \bar{\epsilon}\|y_A^k-y_B^k\|^2 )_{k \in \mathbb{N}^*}$ must converge to  zero. Therefore,  thanks to $\lim_{k\to + \infty}\|y_A^{k-1}-y_B^{k-1}\|=0$ we have
\begin{equation*}
\lim_{j\to +\infty} \|x^{k}-{x^*}\|^2 = \lim_{j\to +\infty}(\|x^{k}-{x^*}\|^2 + \bar{\epsilon}\|y_A^{k-1}-y_B^{k-1}\|^2 )=0, 
\end{equation*}
which implies  that  $\lim_{j\to +\infty} x^{k}={x^*}$,  and the proof of the first statement  follows. We proceed to prove  the second one.  Since  the sequence $\{x^k\}_{k\in\N}$ converges to $x^*\in Fix(T_{A,B})$, it follows  from \eqref{eq2:xk+1)} that $\lim_{k\to \infty}\|y_A^k-y_B^k\|=0$.  Furthermore,  by continuity of the projection,  the  sequence $\{P_A(x^k)\}_{k\in\N^*}$ is convergent to $P_A(x^*)$. Using Proposition~\ref{prop:dist} with $C=A$,  $u=x^k$, $v=y_B^{k-1}$, $\epsilon= \epsilon_k$,  $y_C=y_A^{k-1}$ and $y^+_C=y_A^{k}$, we have that $\|y_A^k-P_A(x^k)\|\leq \sqrt{2\epsilon_k}\|y_A^{k-1}-y_B^{k-1}\|$.  Hence, due to $\lim_{k\to \infty}\|y_A^k-y_B^k\|=0$, we obtain $\lim_{k\to\infty}\|y_A^k-P_A(x^k)\|=0$. Since $\{P_A(x^k)\}_{k\in\N}$ converges to $P_A(x^*)$, we have $\lim_{k\to\infty}y_A^k=P_A(x^*)$, and using Proposition~\ref{pr:cfp} the second  statement  follows, and the proof is completed. 
\end{proof}
We end this section with a special case of Algorithm~\ref{Alg:IDRA}, namely, when the {\it exact projection onto $B$  is easy to compute}.  In this case, we use Algorithm~\ref{Alg:CondG} only to compute inexact project onto the set $A$. We recall that the inexact projection also  accepts an exact one, and  then the  statement of the result is as follows.
\begin{corollary}\label{cr:scase}
Let $\{x^k\}_{k \in \mathbb{N}^*}$ be  a sequence generated by Algorithm~\ref{Alg:IDRA}. Assume that $A\cap B\neq \emptyset$,  $A \subset \mathbb{R}^n$ is a compact set, $0\leq 2\epsilon_k\leq \bar{\epsilon}<1$ and $\bar{\epsilon}\neq 0$.  Furthermore, assume that  $y^k_B= {\PJ}_B(2y^k_A-x^k)$, for all $k \in \mathbb{N}^*$. Then,  $\{x^k\}_{k \in \mathbb{N}^*}$ converges to a point  $x^*\in {\Fix}{\DR}_{A,B}$. As a consequence,   $\{y_A^k\}_{k\in\N}$ converges to  $P_A(x^*)\in A\cap B$.
\end{corollary}
\begin{proof}
It follows from  Remark~\ref{re:eipp} that ${\PJ}_B(2y^k_A-x^k)\in {\PJ}^{\delta}_B\big(y^{k-1}_B,y^{k-1}_A, 2y^k_A-x^k\big)$, for all $\delta \geq 0$.
Thus, the  proof is an immediate consequence of   Theorem~\ref{th:conver} by taking $\delta_k=0$, for all $k \in \mathbb{N}^*$.
\end{proof}
%%%%%%%%%%%%%%%%%%%%%%%%%%%%%
\section{Numerical experiments}  \label{Sec:NumExp}
We applied the approximate Douglas Rachford algorithm to find the interesection of two ellipses, and of an ellipse and a hyperplane. In each case, we consider the cases when the interior of the intersection is empty or not. In the latter case, the iterates $x_k$ are predicted by Theorem \ref{th:conver} to converge to a point in the intersection, whereas in the former case, the approximate projections $y_A^k$ are predicted to converge to a point in the intersection. The algorithms were implemented in Julia and are available at \url{https://github.com/ugonj/approximateDR}.

We perform four sets of experiments. In each experiment, we apply the approximate Douglas-Rachford algorithm to find the intersection of an ellipse with either an ellipse or a half-space, in both cases testing the scenarios when the interior of the intersection is either empty or nonempty.

Define the matrices $R$ and $D$ as follows 
 $$
  R(\theta) := \begin{pmatrix} \cos(\theta) & \sin(\theta) \\ -\sin(\theta) & \cos(\theta)\end{pmatrix}, \qquad \quad D(\alpha,\beta) := \begin{pmatrix}\alpha & 0 \\ 0 & \beta\end{pmatrix}.
  $$

The first ellipse is given by:
\[
    E_1 := \{x\in \R^2: \langle (x-z_1), M_1(x-z_1)\rangle \leq 1\}
\]
where $z_1 = (0,0)$ and
 \[
 M_1 := R\Big(\frac{\pi}{3}\Big)^T D(2,0.2) R\Big(\frac{\pi}{3}\Big), \hfill M_2 := \lambda R\Big(\frac{\pi}{4}\Big)^T D(2,0.2) R\Big(\frac{\pi}{4}\Big), 
 \]
 
 We then find the intersection between the ellipse $E_1$ and the following sets:
 \begin{align*}
     E_2 &:= \{x\in \R^2: \langle (x-z_2), M_2(x-z_2)\rangle \leq 1\}, \\
     E_3 &:= \{x\in \R^2: \langle (x-z_3), M_3(x-z_3)\rangle \leq 1\}, \\
     H_1 &:= \{x\in \R^2: \langle (1,0)^T,x\rangle \geq 1.3\}, \\
     H_2  &:= \{x\in \R^2: \langle (1,0)^T,x\rangle \geq \max_{x\in E_1}(x_1)\}
 \end{align*}

We set $z_2 = (2.3,1.5)$ and
\[
M_2 := \lambda R\Big(\frac{\pi}{4}\Big)^T D(2,0.2) R\Big(\frac{\pi}{4}\Big),
\]
and $z_1\approx (1.788,-0.664)$ and  $M_3 := \lambda M_2$, with $\lambda\approx 0.145$, so that the intersection of the two sets is a single point ($x\approx (1.393,1.310)$).

All sets intersect the ellipse $E_1$, and $\mathrm{int} (E_1\cap E_2) \neq \emptyset$, $\mathrm{int} (E_1\cap E_3) = \emptyset$, $\mathrm{int} (E_1\cap H_1) \neq \emptyset$, and $\mathrm{int} (E_1\cap H_2) = \emptyset$.

In all cases, we set the initial point to $x_0= (-1,1.5)$. We stop the algorithm when $\|y_A^k-y_B^k\|^2 < 10^{-6}$ (note that $x^k\in {\Fix}{\DR}_{A,B}$ iff  $\|y_A^k-y_B^k\|^2 =0$). In Tables~\ref{table:iterations} and ~\ref{table:computation}, we summarise the result of these experiments. We report the number of iterations taken by the algorithm in Table~\ref{table:iterations}, namely the number of iterations taken by the sequence $\{x^k\}$ to converge. When the intersection is nonempty, we also report in brackets the number of iterations until the shadow sequence $\{y_A^k\}$ or $\{y_B^k\}$ reaches the intersection (so we have an alternative stopping criterion reached whenever we find a point in the intersection). The shadow sequences do not reach the intersection when the interior is empty (although they converge to it). We present the relative CPU time taken by the Approximate Douglas Rachford method in Table~\ref{table:computation}. This number is the ratio of the CPU time taken by the appDR method over the CPU time taken by the exact DR, both averaged over 10000 runs, as measured using BenchmarkTools.jl~\cite{benchmarktools} (numbers less than 1 mean that the approximate DR method was faster than the exact DR method). To simulate the case when a closed form for the projection isn't known, we use the Conditional Gradient algorithm to find the exact projections over the ellipses, and stop when $-s_\ell^*\leq 10^{-6}$.
In all experiments we project onto half-spaces using a closed form formula for an exact projection. By Corollary~\ref{cr:scase}, the results from this paper apply.

The results presented in these tables indicate that the approximate Douglas Rachford method is competitive with the exact Douglas Rachford method. In all cases the approximate version of the method is faster than the exact version, even when it requires more iterations. In all cases both algorithms require few iterations to reach a point in the intersection, even when it takes more iterations for the sequences to converge to a fix point of the operator. 
  \begin{table}[ht]
  \centering
  \begin{tabular}{cccccc}\toprule
  $\epsilon$ & $E_2$ & $E_3$ & $H_1$ & $H_2$ \\\midrule
  0.245  & 52 (2) & 14 & 10 (4) & 5\\
  0.120  & 51 (2) & 13 & 12 (3) & 5\\
  Exact  & 24 (6) & 11 & 15 (4) & 6\\\bottomrule
  \end{tabular}
  \caption{Number of iterations taken by the approximate Douglas Rachford algorithm before convergence, for various values of $\epsilon$.}\label{table:iterations}
  \end{table}
  
    \begin{table}[ht]
  \centering
  \begin{tabular}{cccccc}\toprule
  $\epsilon$ & $E_2$ & $E_3$ & $H_1$ & $H_2$ \\\midrule
  0.245  & 0.34 & 0.48 & 0.17 & 0.15\\
  0.120  & 0.36 & 0.48 & 0.23 & 0.18\\\bottomrule
  \end{tabular}
  \caption{Relative computational time taken by the approximate Douglas Rachford algorithm compared to the exact Douglas Rachford algorithm, for various values of $\epsilon$.}\label{table:computation}
  \end{table}
  
  We illustrate the iterations of the algorithms on our experiments in  Figures~\ref{figni} to~\ref{fig:ellipse_hyperplane_nointerior}. It can be seen that the behaviour of the algorithms are comparable for all values of $\epsilon$.
  
  \begin{figure}[ht]
   % \centering
    \begin{tikzpicture}[scale=0.82]
      %\clip (0,0) rectangle (2.3,-1.5);
      \draw[rotate=-45,fill=blue,fill opacity=0.2] (0,0) ellipse (2 and 0.2);
      \draw[rotate around={60:(2.3,0.5)},fill=green,fill opacity=0.2] (2.3,0.5) ellipse (2 and 0.4);
      \draw[every node/.style={circle,draw=blue!50,fill=blue!20,inner sep=0.7}]
(-1.0,1.5) node {} -- (1.2999999999999998,2.0) node {} -- (2.571917656836094,0.7943428328503053) node {} -- (2.9060752191369823,0.4915884333641444) node {} -- (3.053518594044455,0.2398474393217438) node {} -- (3.118774791718569,0.07629942281524071) node {} -- (3.113444995191702,-0.0008496721484581027) node {} -- (3.0630165806437053,-0.013506528475393242) node {} -- (3.012588166095709,-0.026163384802328604) node {} -- (2.9621597515477123,-0.038820241129263966) node {} -- (2.9226302073318142,-0.0672811033693912) node {} -- (2.8831006631159157,-0.09574196560951842) node {} -- (2.8435711189000177,-0.12420282784964565) node {} -- (2.8040415746841196,-0.15266369008977287) node {} -- (2.7645120304682216,-0.1811245523299001) node {} -- (2.7249824862523235,-0.20958541457002733) node {} -- (2.6854529420364255,-0.23804627681015456) node {} -- (2.6459233978205274,-0.2665071390502818) node {} -- (2.6063938536046294,-0.294968001290409) node {} -- (2.5668643093887313,-0.32342886353053624) node {} -- (2.5273347651728333,-0.35188972577066346) node {} -- (2.4878052209569352,-0.3803505880107907) node {} -- (2.448275676741037,-0.4088114502509179) node {} -- (2.408746132525139,-0.43727231249104515) node {} -- (2.369216588309241,-0.4657331747311724) node {} -- (2.329687044093343,-0.4941940369712996) node {} -- (2.290157499877445,-0.5226548992114268) node {} -- (2.250627955661547,-0.551115761451554) node {} -- (2.211098411445649,-0.5795766236916813) node {} -- (2.171568867229751,-0.6080374859318085) node {} -- (2.132039323013853,-0.6364983481719357) node {} -- (2.0925097787979547,-0.664959210412063) node {} -- (2.0529802345820567,-0.6934200726521902) node {} -- (2.0059928570191254,-0.7108175314931706) node {} -- (1.959005479456194,-0.7282149903341509) node {} -- (1.9190430256847728,-0.7557160067052413) node {} -- (1.8790805719133516,-0.7832170230763316) node {} -- (1.8391181181419305,-0.8107180394474218) node {} -- (1.7991556643705096,-0.8382190558185121) node {} -- (1.7591932105990884,-0.8657200721896023) node {} -- (1.7192307568276672,-0.8932210885606926) node {} -- (1.6792683030562459,-0.9207221049317829) node {} -- (1.6393058492848245,-0.9482231213028731) node {} -- (1.599343395513403,-0.9757241376739634) node {} -- (1.5593809417419817,-1.0032251540450536) node {} -- (1.5194184879705603,-1.030726170416144) node {} -- (1.479456034199139,-1.0582271867872342) node {} -- (1.4394935804277176,-1.0857282031583244) node {} -- (1.3995311266562964,-1.1132292195294147) node {} -- (1.3595686728848753,-1.140730235900505) node {} -- (1.319606219113454,-1.1682312522715952) node {} -- (1.3102419246106183,-1.176429562074131) node {};
    \end{tikzpicture}
    \hfill
    \begin{tikzpicture}[scale=0.82]
      %\clip (0,0) rectangle (2.3,-1.5);
     \draw[rotate=-45,fill=blue,fill opacity=0.2] (0,0) ellipse (2 and 0.2);
      \draw[rotate around={60:(2.3,0.5)},fill=green,fill opacity=0.2] (2.3,0.5) ellipse (2 and 0.4);
      \draw[every node/.style={circle,draw=blue!50,fill=blue!20,inner sep=0.7}]
(-1.0,1.5) node {} -- (1.2999999999999998,2.0) node {} -- (2.963169894987529,0.3779866768130681) node {} -- (2.9190527789057064,0.475606803060068) node {} -- (3.0420545305689775,0.23114639959945016) node {} -- (3.046897046300576,0.14022941363200037) node {} -- (3.039108983992874,0.06691329772194932) node {} -- (3.012625198171726,0.019617646954303902) node {} -- (2.9782714509785277,-0.01634503000530918) node {} -- (2.9439177037853295,-0.05230770696492204) node {} -- (2.9010757017157767,-0.0759816359886274) node {} -- (2.858233699646224,-0.09965556501233275) node {} -- (2.815391697576671,-0.1233294940360381) node {} -- (2.78020694361485,-0.15807428345291763) node {} -- (2.745022189653029,-0.19281907286979738) node {} -- (2.7015228720104005,-0.21550600893191296) node {} -- (2.658023554367772,-0.23819294499402854) node {} -- (2.6217515212902485,-0.2713706064320205) node {} -- (2.5854794882127248,-0.30454826787001243) node {} -- (2.5418172645148056,-0.32699196494542404) node {} -- (2.4981550408168864,-0.34943566202083565) node {} -- (2.4610965688114392,-0.3814580987680922) node {} -- (2.424038096805992,-0.4134805355153488) node {} -- (2.386979624800545,-0.4455029722626054) node {} -- (2.34270468200004,-0.46700210860909674) node {} -- (2.298429739199535,-0.4885012449555881) node {} -- (2.260644987746046,-0.5194110691500953) node {} -- (2.222860236292557,-0.5503208933446024) node {} -- (2.185075484839068,-0.5812307175391096) node {} -- (2.140977075490799,-0.602970121922823) node {} -- (2.09687866614253,-0.6247095263065363) node {} -- (2.058759054345253,-0.6550613878667673) node {} -- (2.020639442547976,-0.6854132494269981) node {} -- (1.9825198307506993,-0.7157651109872292) node {} -- (1.9444002189534224,-0.7461169725474602) node {} -- (1.9008843790208132,-0.768624542362814) node {} -- (1.857368539088204,-0.7911321121781679) node {} -- (1.8187004518330383,-0.8206026129214767) node {} -- (1.7800323645778724,-0.8500731136647857) node {} -- (1.7413642773227065,-0.8795436144080946) node {} -- (1.7026961900675404,-0.9090141151514035) node {} -- (1.6640281028123742,-0.9384846158947124) node {} -- (1.625360015557208,-0.9679551166380214) node {} -- (1.5826488604123463,-0.9914399769041613) node {} -- (1.539937705267485,-1.0149248371703012) node {} -- (1.4972265501226236,-1.0384096974364412) node {} -- (1.4577886002688705,-1.0664071947478975) node {} -- (1.4183506504151175,-1.0944046920593538) node {} -- (1.3789127005613646,-1.12240218937081) node {} -- (1.3394747507076117,-1.1503996866822663) node {} -- (1.3086040191168975,-1.1736789845792477) node {};
    \end{tikzpicture}
    \hfill
    \begin{tikzpicture}[scale=0.82]
      %\clip (0,0) rectangle (2.3,-1.5);
      \draw[rotate=-45,fill=blue,fill opacity=0.2] (0,0) ellipse (2 and 0.2);
      \draw[rotate around={60:(2.3,0.5)},fill=green,fill opacity=0.2] (2.3,0.5) ellipse (2 and 0.4);
      \draw[every node/.style={circle,draw=blue!50,fill=blue!20,inner sep=0.7}]
(-1.0,1.5) node {} -- (1.6642919109341463,0.05325872279447874) node {} -- (2.032996183271009,-0.44406258824826483) node {} -- (2.074196613943129,-0.5824739540863599) node {} -- (2.0567396398979265,-0.6425376482340701) node {} -- (2.0237910994104022,-0.6806963940912838) node {} -- (1.9859668407755773,-0.7118297790013286) node {} -- (1.9465312735824996,-0.7406263295966204) node {} -- (1.9065611265030369,-0.7686458518630948) node {} -- (1.866415610183123,-0.7964102731761675) node {} -- (1.8262134720606285,-0.8240923080786322) node {} -- (1.7859933549462594,-0.8517481806961695) node {} -- (1.745767618586219,-0.8793958761806278) node {} -- (1.705540150259962,-0.9070410512834426) node {} -- (1.6653121540171032,-0.934685458152309) node {} -- (1.6250839980285108,-0.9623296325560127) node {} -- (1.584855793799114,-0.9899737367699168) node {} -- (1.5446275749363287,-1.0176178197067063) node {} -- (1.5043993515859129,-1.045261896084877) node {} -- (1.4641711268226756,-1.0729059704063202) node {} -- (1.4239429015981622,-1.100550044056583) node {} -- (1.3837146762064947,-1.1281941174824694) node {} -- (1.3434864507595965,-1.1558381908090352) node {} -- (1.311846900103361,-1.177580289711558) node {} -- (1.311846900103361,-1.177580289711558) node {};
    \end{tikzpicture}
      \caption{Iterates of the approximate Douglas Rachford algorithm to find the intersection with nonempty interior of two ellipses for $\epsilon=0.245$, $\epsilon=0.120$ and exact projections.} \label{figni}
   \end{figure}

  \begin{figure}[ht]
   % \centering
    \begin{tikzpicture}[scale=1.1]
      %\clip (0,0) rectangle (2.3,-1.5);
      \draw[rotate=-45,fill=blue,fill opacity=0.2] (0,0) ellipse (2 and 0.2);
      \draw[rotate around={60:(1.7882409,-0.66441136)},fill=green,fill opacity=0.2] (1.7882409,-0.66441136) ellipse (0.762734 and 0.152547);
      \draw[every node/.style={circle,draw=blue!50,fill=blue!20,inner sep=0.7}]
(-1.0,1.5) node {} -- (0.7882409346931343,0.8355886399695496) node {} -- (2.1732568289547194,-0.4398976995074002) node {} -- (2.2372997882250685,-0.41182019608840004) node {} -- (2.351701734237025,-0.5913114360687454) node {} -- (2.401745726166118,-0.669597081029015) node {} -- (2.4253326885721513,-0.7112757109014962) node {} -- (2.430023670117614,-0.720292225836492) node {} -- (2.438718442852827,-0.736774260235086) node {} -- (2.44741321558804,-0.75325629463368) node {} -- (2.44824627058572,-0.7548223367569835) node {} -- (2.4507993880042847,-0.7597700875168598) node {} -- (2.4517006141245905,-0.7615280952636029) node {} -- (2.4529153961255297,-0.7638957892198579) node {};
      \draw[every node/.style={circle,draw=red!50,fill=red!20,inner sep=0.7}]
        (-1.0,1.5) -- (0.0,0.0) node {}
        (0.7882409346931343,0.8355886399695496) -- (0.0,0.0) node {}
        (2.1732568289547194,-0.4398976995074002) -- (1.3209729349912358,-1.30356384289595) node {}
        (2.2372997882250685,-0.41182019608840004) -- (1.2791902728162292,-1.1312037677083835) node {}
        (2.351701734237025,-0.5913114360687454) -- (1.3435482268990935,-1.2324093627284591) node {}
        (2.401745726166118,-0.669597081029015) -- (1.370005256422153,-1.2690163778162478) node {}
        (2.4253326885721513,-0.7112757109014962) -- (1.3889012372827236,-1.301678492753733) node {}
        (2.430023670117614,-0.720292225836492) -- (1.3848974460929733,-1.2942129732901348) node {}
        (2.438718442852827,-0.736774260235086) -- (1.3848974460929733,-1.2942129732901348) node {}
        (2.44741321558804,-0.75325629463368) -- (1.3927591638305061,-1.3091289655654255) node {}
        (2.44824627058572,-0.7548223367569835) -- (1.3905434426206538,-1.304805959559421) node {}
        (2.4507993880042847,-0.7597700875168598) -- (1.3921953339189121,-1.3079957025725544) node {}
        (2.4517006141245905,-0.7615280952636029) -- (1.3918817780382788,-1.3073860163630424) node {}
        (2.4529153961255297,-0.7638957892198579) -- (1.3926497788587973,-1.308880159879567) node {};
    \end{tikzpicture}
    \hfill
    \begin{tikzpicture}[scale=1.1]
      %\clip (0,0) rectangle (2.3,-1.5);
     \draw[rotate=-45,fill=blue,fill opacity=0.2] (0,0) ellipse (2 and 0.2);
      \draw[rotate around={60:(1.7882409,-0.66441136)},fill=green,fill opacity=0.2] (1.7882409,-0.66441136) ellipse (0.762734 and 0.152547);
      \draw[every node/.style={circle,draw=blue!50,fill=blue!20,inner sep=0.7}]
      (-1.0,1.5) node {} -- (0.7882409346931343,0.8355886399695496) node {} -- (2.1732568289547194,-0.4398976995074002) node {} -- (2.2372997882250685,-0.41182019608840004) node {} -- (2.358082781450654,-0.5997522290806182) node {} -- (2.403518905584,-0.6689146235604573) node {} -- (2.4262694933707945,-0.7094298131320186) node {} -- (2.439320099641045,-0.7336224310495805) node {} -- (2.446872504568434,-0.7479446644832588) node {} -- (2.451231243501812,-0.7563251442317466) node {} -- (2.4537300435938496,-0.7611688979746194) node {} -- (2.4551504428917017,-0.76393508059268) node {} -- (2.4559505826632764,-0.7654971200554284) node {};
      \draw[every node/.style={circle,draw=red!50,fill=red!20,inner sep=0.7}]
        (-1.0,1.5) -- (0.0,0.0) node {}
        (0.7882409346931343,0.8355886399695496) -- (0.0,0.0) node {}
        (2.1732568289547194,-0.4398976995074002) -- (1.3209729349912358,-1.30356384289595) node {}
        (2.2372997882250685,-0.41182019608840004) -- (1.2723786900111516,-1.1219481969480503) node {}
        (2.358082781450654,-0.5997522290806182) -- (1.347725559103391,-1.2407178354604291) node {}
        (2.403518905584,-0.6689146235604573) -- (1.3704110954499422,-1.2693650403687071) node {}
        (2.4262694933707945,-0.7094298131320186) -- (1.3801110769664862,-1.2856876120227065) node {}
        (2.439320099641045,-0.7336224310495805) -- (1.385609278309348,-1.29555799650659) node {}
        (2.446872504568434,-0.7479446644832588) -- (1.3888029443033592,-1.3014997501917807) node {}
        (2.451231243501812,-0.7563251442317466) -- (1.3906628831446994,-1.3050364761973954) node {}
        (2.4537300435938496,-0.7611688979746194) -- (1.391741283938885,-1.3071140473222078) node {}
        (2.4551504428917017,-0.76393508059268) -- (1.3923615434651626,-1.30831819047752) node {}
        (2.4559505826632764,-0.7654971200554284) -- (1.3927148811277472,-1.3090071858570043) node {};
    \end{tikzpicture}
    \hfill
    \begin{tikzpicture}[scale=1.1]
      %\clip (0,0) rectangle (2.3,-1.5);
      \draw[rotate=-45,fill=blue,fill opacity=0.2] (0,0) ellipse (2 and 0.2);
      \draw[rotate around={60:(1.7882409,-0.66441136)},fill=green,fill opacity=0.2] (1.7882409,-0.66441136) ellipse (0.762734 and 0.152547);
      \draw[every node/.style={circle,draw=blue!50,fill=blue!20,inner sep=0.7}]
        (-1.0,1.5) node {} -- (1.796407341300168,-0.3928064083779881) node {} -- (2.0562594404161945,-0.7506648026887706) node {} -- (2.1100986840504277,-0.8416857397464588) node {} -- (2.130963651718051,-0.8796862556839584) node {} -- (2.1409939185290296,-0.8985846360141039) node {} -- (2.146270516572635,-0.9087006549364485) node {} -- (2.1491706608241725,-0.9143132108210776) node {} -- (2.1508017761356104,-0.9174864654568593) node {} -- (2.151730812606964,-0.919299246128854) node {} -- (2.1522637286462025,-0.9203408692771082) node {} -- (2.152570656068578,-0.9209413688463102) node {};
      \draw[every node/.style={circle,draw=red!50,fill=red!20,inner sep=0.7}]
        (-1.0,1.5) -- (-1.1480148168230162,1.291200937154192) node {}
        (1.796407341300168,-0.3928064083779881) -- (1.1271527216378239,-0.9333968013887954) node {}
        (2.0562594404161945,-0.7506648026887706) -- (1.3387946067315095,-1.2178076221779803) node {}
        (2.1100986840504277,-0.8416857397464588) -- (1.3723048214231244,-1.2718954001367297) node {}
        (2.130963651718051,-0.8796862556839584) -- (1.3831552192640835,-1.2910278732991143) node {}
        (2.1409939185290296,-0.8985846360141039) -- (1.3878682496402852,-1.2997315879826448) node {}
        (2.146270516572635,-0.9087006549364485) -- (1.3902097392230195,-1.3041671108446713) node {}
        (2.1491706608241725,-0.9143132108210776) -- (1.391455426462563,-1.3065609138292777) node {}
        (2.1508017761356104,-0.9174864654568593) -- (1.3921430708418332,-1.307893194545107) node {}
        (2.151730812606964,-0.919299246128854) -- (1.3925305458235786,-1.3086474448665046) node {}
        (2.1522637286462025,-0.9203408692771082) -- (1.3927514356939765,-1.309078589857765) node {}
        (2.152570656068578,-0.9209413688463102) -- (1.3928781996775093,-1.3093264018970017) node {};
    \end{tikzpicture}
       \caption{Iterates of the approximate Douglas Rachford algorithm to find the intersection with empty interior of two ellipses for $\epsilon=0.245$, $\epsilon=0.120$ and exact projections.} \label{figei}
   \end{figure}

   \begin{figure}
    \begin{tikzpicture}[every mark/.append style={scale=0.6}]
     \begin{axis}[xlabel={\small Iteration (\(k\))},ylabel={\small $\|y_A^k-y_B^k\|$},width=0.3\textwidth]
       \addplot coordinates {(1,2.3537204591879637) (2,1.7525363706556414) (3,0.45091185707530235) (4,0.2917414555485072) (5,0.1760861296016655) (6,0.07733297863612668) (7,0.051992509132617076) (8,0.051992509132617076) (9,0.051992509132617076) (10,0.04870939894279292) (11,0.04870939894279292) (12,0.04870939894279292) (13,0.04870939894279292) (14,0.04870939894279292) (15,0.04870939894279292) (16,0.04870939894279292) (17,0.04870939894279292) (18,0.04870939894279292) (19,0.04870939894279292) (20,0.04870939894279292) (21,0.04870939894279292) (22,0.04870939894279292) (23,0.04870939894279292) (24,0.04870939894279292) (25,0.04870939894279292) (26,0.04870939894279292) (27,0.04870939894279292) (28,0.04870939894279292) (29,0.04870939894279292) (30,0.04870939894279292) (31,0.04870939894279292) (32,0.04870939894279292) (33,0.05010474253366705) (34,0.05010474253366705) (35,0.04851086077236671) (36,0.04851086077236671) (37,0.04851086077236671) (38,0.04851086077236671) (39,0.04851086077236671) (40,0.04851086077236671) (41,0.04851086077236671) (42,0.04851086077236671) (43,0.04851086077236671) (44,0.04851086077236671) (45,0.04851086077236671) (46,0.04851086077236671) (47,0.04851086077236671) (48,0.04851086077236671) (49,0.04851086077236671) (50,0.04851086077236671) (51,0.0124459750583952) (52,0.0)};
     \end{axis}
    \end{tikzpicture}
    \hfill
    \begin{tikzpicture}[every mark/.append style={scale=0.6}]
     \begin{axis}[xlabel={\small Iteration (\(k\))},ylabel={\small $\|y_A^k-y_B^k\|$},width=0.3\textwidth]
       \addplot coordinates {(1,2.3537204591879637) (2,2.3231576184556966) (3,0.10712613583928612) (4,0.27366095770562715) (5,0.09104585820352434) (6,0.07372860209338492) (7,0.054205806819523535) (8,0.04973423449009408) (9,0.04973423449009408) (10,0.04894785037922489) (11,0.04894785037922489) (12,0.04894785037922489) (13,0.04944863297379618) (14,0.04944863297379618) (15,0.049060041818783456) (16,0.049060041818783456) (17,0.04915707072305052) (18,0.04915707072305052) (19,0.04909286421324379) (20,0.04909286421324379) (21,0.04897720697028921) (22,0.04897720697028921) (23,0.04897720697028921) (24,0.04921873041468084) (25,0.04921873041468084) (26,0.04881705310787681) (27,0.04881705310787681) (28,0.04881705310787681) (29,0.04916575444357705) (30,0.04916575444357705) (31,0.04872720291322409) (32,0.04872720291322409) (33,0.04872720291322409) (34,0.04872720291322409) (35,0.048992030209346234) (36,0.048992030209346234) (37,0.048618220720574465) (38,0.048618220720574465) (39,0.048618220720574465) (40,0.048618220720574465) (41,0.048618220720574465) (42,0.048618220720574465) (43,0.04874198842403286) (44,0.04874198842403286) (45,0.04874198842403286) (46,0.048365398213724375) (47,0.048365398213724375) (48,0.048365398213724375) (49,0.048365398213724375) (50,0.03866429592689254) (51,0.0)};
     \end{axis}
    \end{tikzpicture}
    \hfill
    \begin{tikzpicture}[every mark/.append style={scale=0.6}]
     \begin{axis}[xlabel={\small Iteration (\(k\))},ylabel={\small $\|y_A^k-y_B^k\|$},width=0.3\textwidth]
       \addplot coordinates {(1,3.0317506015237115) (2,0.6190891105945319) (3,0.14441323236024872) (4,0.06254912707212242) (5,0.05041523783196307) (6,0.048989409031250465) (7,0.0488303725771877) (8,0.048812972513414536) (9,0.04881112139095212) (10,0.04881093080408765) (11,0.04881091180188092) (12,0.048810909961900996) (13,0.048810909788748066) (14,0.04881090977278441) (15,0.048810909771329125) (16,0.04881090977754143) (17,0.04881090978754908) (18,0.0488109097716184) (19,0.048810909771186954) (20,0.04881090977123197) (21,0.048810909781918425) (22,0.048810909771187516) (23,0.038389842802694626) (24,0.0)};
     \end{axis}
    \end{tikzpicture}
      \caption{Distance $\|y_A^k-y_B^k\|$, when $\mathrm{int}(E_1\cap E_2)\neq \emptyset$ for $\epsilon=0.245$, $\epsilon=0.120$ and exact projections}
   \end{figure}

   \begin{figure}
    \begin{tikzpicture}[every mark/.append style={scale=0.6}]
     \begin{axis}[xlabel={\small Iteration (\(k\))},ylabel={\small $\|y_A^k-y_B^k\|$},width=0.3\textwidth]
       \addplot coordinates {(1,1.907681340226844) (2,1.882852737085279) (3,0.0699274397525575) (4,0.2128495019515072) (5,0.09291417187407157) (6,0.04789000922521213) (7,0.010163800934339186) (8,0.018634820386401376) (9,0.018634820386401376) (10,0.0017738287857401972) (11,0.005567642780821396) (12,0.001975550494806389) (13,0.0026611407291407562) (14,0.00098117470102488)};
     \end{axis}
    \end{tikzpicture}
    \hfill
    \begin{tikzpicture}[every mark/.append style={scale=0.6}]
     \begin{axis}[xlabel={\small Iteration (\(k\))},ylabel={\small $\|y_A^k-y_B^k\|$},width=0.3\textwidth]
       \addplot coordinates {(1,1.907681340226844) (2,1.882852737085279) (3,0.0699274397525575) (4,0.22339870294413075) (5,0.08275190744898708) (6,0.04646579204817453) (7,0.02748819902660491) (8,0.01619151601043035) (9,0.009446218603449534) (10,0.0054503167083723505) (11,0.003109549877683062) (12,0.001755047274929662) (13,0.0007274810327096371)};
     \end{axis}
    \end{tikzpicture}
    \hfill
    \begin{tikzpicture}[every mark/.append style={scale=0.6}]
     \begin{axis}[xlabel={\small Iteration (\(k\))},ylabel={\small $\|y_A^k-y_B^k\|$},width=0.3\textwidth]
       \addplot coordinates {(1,3.376775106232906) (2,0.44225077025788967) (3,0.10575195098893429) (4,0.04335188677886042) (5,0.02139521047810123) (6,0.011409484026502213) (7,0.006317564422925208) (8,0.0035679240662769658) (9,0.002036978775015748) (10,0.001170033455870412) (11,0.0006743917075539298)};
     \end{axis}
    \end{tikzpicture}

      \caption{Distance $\|y_A^k-y_B^k\|$, when $\mathrm{int}(E_1\cap E_2) = \emptyset$ for $\epsilon=0.245$, $\epsilon=0.120$ and exact projections}
   \end{figure}

  \begin{figure}[ht]
    \begin{tikzpicture}[scale=1.1]
    %\clip (0,0) rectangle (2.3,-1.5);
    \draw[rotate=-45,fill=blue,fill opacity=0.2] (0,0) ellipse (2 and 0.2);
    \fill[green,fill opacity=0.2] (1.3,1.6) rectangle (1.8,-1.6);
    \draw (1.3,1.6) -- (1.3,-1.6);
    \draw[every node/.style={circle,draw=blue!50,fill=blue!20,inner sep=0.6}]
      (-1.0,1.5) node {} -- (0.30000000000000004,0.0) node {} -- (1.6,0.0) node {} -- (2.0840015998400157,-0.7998400159984002) node {} -- (1.9627345594848262,-1.393123138565978) node {} -- (1.8414675191296366,-1.3931231385659775) node {} -- (1.720200478774447,-1.3931231385659777) node {} -- (1.5989334384192575,-1.393123138565978) node {} -- (1.4776663980640679,-1.3931231385659775) node {} -- (1.4212670403551897,-1.3931231385659777) node {};
  \end{tikzpicture}
  \hfill
    \begin{tikzpicture}[scale=1.1]
    %\clip (0,0) rectangle (2.3,-1.5);
    \draw[rotate=-45,fill=blue,fill opacity=0.2] (0,0) ellipse (2 and 0.2);
    \fill[green,fill opacity=0.2] (1.3,1.6) rectangle (1.8,-1.6);
    \draw (1.3,1.6) -- (1.3,-1.6);
    \draw[every node/.style={circle,draw=blue!50,fill=blue!20,inner sep=0.6}]
      (-1.0,1.5) node {} -- (0.30000000000000004,0.0) node {} -- (1.6,0.0) node {} -- (2.1841021525042668,-0.6671968283947391) node {} -- (2.121754667496822,-1.3188203935401046) node {} -- (2.059407182489377,-1.3188203935401046) node {} -- (1.9381401421341875,-1.3931231385659775) node {} -- (1.8168731017789979,-1.3931231385659777) node {} -- (1.6956060614238084,-1.393123138565978) node {} -- (1.5743390210686188,-1.3931231385659775) node {} -- (1.4530719807134291,-1.3931231385659777) node {} -- (1.4212670403551897,-1.393123138565978) node {};
  \end{tikzpicture}
  \hfill 
    \begin{tikzpicture}[scale=1.1]
    %\clip (0,0) rectangle (2.3,-1.5);
    \draw[rotate=-45,fill=blue,fill opacity=0.2] (0,0) ellipse (2 and 0.2);
    \fill[green,fill opacity=0.2] (1.3,1.6) rectangle (1.8,-1.6);
    \draw (1.3,1.6) -- (1.3,-1.6);
    \draw[every node/.style={circle,draw=blue!50,fill=blue!20,inner sep=0.6}]
      (-1.0,1.5) node {} -- (1.448014816823016,1.291200937154192) node {} -- (2.534644192872839,0.06910392569820178) node {} -- (2.6570362611439555,-0.9967120306975277) node {} -- (2.5406323019877544,-1.3650944442802646) node {} -- (2.41938182423479,-1.3917482010974944) node {} -- (2.2981148314238724,-1.3930502954533113) node {} -- (2.1768477912133575,-1.3931187850445261) node {} -- (2.0555807508587907,-1.3931228395123516) node {} -- (1.9343137105036063,-1.3931231144132972) node {} -- (1.813046670148417,-1.3931231361994252) node {} -- (1.6917796297931984,-1.393123138271285) node {} -- (1.5705125894380092,-1.3931231385156368) node {} -- (1.4492455490828198,-1.3931231385522904) node {} -- (1.4212670403552037,-1.3931231385568696) node {} -- (1.4212670403552037,-1.3931231385568696) node {};
  \end{tikzpicture}
  \caption{Iterates of the approximate Douglas Rachford algorithm to find the intersection with nonempty interior of  ellipse and a half-plane for $\epsilon=0.245$, $\epsilon=0.120$ and exact projections.}\label{fig:ellipsehyperplane nonempty}
\end{figure}

   \begin{figure}
    \begin{tikzpicture}[every mark/.append style={scale=0.6}]
     \begin{axis}[xlabel={\small Iteration (\(k\))},ylabel={\small $\|y_A^k-y_B^k\|$},width=0.3\textwidth]
       \addplot coordinates {(1,1.9849433241279208) (2,1.3) (3,0.9348805270407635) (4,0.605549798612833) (5,0.12126704035518965) (6,0.12126704035518965) (7,0.12126704035518965) (8,0.12126704035518965) (9,0.05639935770887816) (10,0.0)};
     \end{axis}
    \end{tikzpicture}
    \hfill
    \begin{tikzpicture}[every mark/.append style={scale=0.6}]
     \begin{axis}[xlabel={\small Iteration (\(k\))},ylabel={\small $\|y_A^k-y_B^k\|$},width=0.3\textwidth]
       \addplot coordinates {(1,1.9849433241279208) (2,1.3) (3,0.8867507724158556) (4,0.6545994802468986) (5,0.0623474850074448) (6,0.142220227094767) (7,0.12126704035518965) (8,0.12126704035518965) (9,0.12126704035518965) (10,0.12126704035518965) (11,0.03180494035823944) (12,0.0)};
     \end{axis}
    \end{tikzpicture}
    \hfill
    \begin{tikzpicture}[every mark/.append style={scale=0.6}]
     \begin{axis}[xlabel={\small Iteration (\(k\))},ylabel={\small $\|y_A^k-y_B^k\|$},width=0.3\textwidth]
       \addplot coordinates {(1,2.456903252476644) (2,1.6353239759460785) (3,1.072820335043769) (4,0.3863359734017545) (5,0.12414548363832685) (6,0.12127398317493569) (7,0.12126705955140005) (8,0.12126704042234578) (9,0.12126704035549613) (10,0.12126704035519138) (11,0.1212670403552183) (12,0.12126704035518965) (13,0.12126704035518965) (14,0.027978508727615203) (15,0.0)};
     \end{axis}
    \end{tikzpicture}
      \caption{Distance $\|y_A^k-y_B^k\|$, when $\mathrm{int}(E_1\cap H) \neq \emptyset$ for $\epsilon=0.245$, $\epsilon=0.120$ and exact projections.}
   \end{figure}

\begin{figure}
    \begin{tikzpicture}[scale=1]
    %\clip (0,0) rectangle (2.3,-1.5);
    \draw[rotate=-45,fill=blue,fill opacity=0.2] (0,0) ellipse (2 and 0.2);
    \fill[green,fill opacity=0.2] (1.421267,1.6) rectangle (1.9,-1.6);
    \draw (1.421267,1.6) -- (1.421267,-1.6);
    \draw[every node/.style={circle,draw=blue!50,fill=blue!20,inner sep=0.7}]
    (-1.0,1.5) node {} -- (0.4212670403551897,0.0) node {} -- (1.8425340807103794,0.0) node {} -- (2.3241105822531214,-0.921082805370617) node {} -- (2.3241105822531214,-1.3931231385659777) node {};
    \draw[dashed,very thin,every node/.style={circle,draw=red!50,fill=red!20,inner sep=0.7}]
      (-1.0,1.5) -- (0.0,0.0) node {}
      (0.4212670403551897,0.0) -- (0.0,0.0) node {}
      (1.8425340807103794,0.0) -- (0.9396905388124479,-0.921082805370617) node {}
      (2.3241105822531214,-0.921082805370617) -- (1.4212670403551897,-1.3931231385659775) node {}
      (2.3241105822531214,-1.3931231385659777) -- (1.4212670403551897,-1.3931231385659775) node {};
  \end{tikzpicture} 
  \hfill  
    \begin{tikzpicture}[scale=1]
    %\clip (0,0) rectangle (2.3,-1.5);
    \draw[rotate=-45,fill=blue,fill opacity=0.2] (0,0) ellipse (2 and 0.2);
    \fill[green,fill opacity=0.2] (1.421267,1.6) rectangle (1.9,-1.6);
    \draw (1.421267,1.6) -- (1.421267,-1.6);
    \draw[every node/.style={circle,draw=blue!50,fill=blue!20,inner sep=0.7}]
    (-1.0,1.5) node {} -- (0.4212670403551897,0.0) node {} -- (1.6276883113541505,-0.21059139768580848) node {} -- (2.1940199611629523,-0.8137977761804803) node {} -- (2.1940199611629523,-1.3931231385659775) node {};
    \draw[dashed,very thin,every node/.style={circle,draw=red!50,fill=red!20,inner sep=0.7}]
      (-1.0,1.5) -- (0.0,0.0) node {}
      (0.4212670403551897,0.0) -- (0.2148457693562289,-0.21059139768580848) node {}
      (1.6276883113541505,-0.21059139768580848) -- (0.8549353905463879,-0.8137977761804804) node {}
      (2.1940199611629523,-0.8137977761804803) -- (1.42126704035519,-1.3931231385659777) node {}
      (2.1940199611629523,-1.3931231385659775) -- (1.42126704035519,-1.3931231385659777) node {};
  \end{tikzpicture} 
  \hfill 
   \begin{tikzpicture}[scale=1]
    %\clip (0,0) rectangle (2.3,-1.5);
    \draw[rotate=-45,fill=blue,fill opacity=0.2] (0,0) ellipse (2 and 0.2);
    \fill[green,fill opacity=0.2] (1.421267,1.6) rectangle (1.9,-1.6);
    \draw (1.421267,1.6) -- (1.421267,-1.6);
    \draw[every node/.style={circle,draw=blue!50,fill=blue!20,inner sep=0.7}]
(-1.0,1.5) node {} -- (1.569281857178206,1.291200937154192) node {} -- (2.72232299125678,0.013467153235845908) node {} -- (2.897121701603159,-1.0865632432073675) node {} -- (2.8987833405197327,-1.3778892696171487) node {} -- (2.8987861425969155,-1.3925615307061925) node {} -- (2.898786146313372,-1.3931027146894364) node {};
    \draw[dashed,very thin,every node/.style={circle,draw=red!50,fill=red!20,inner sep=0.7}]
(-1.0,1.5) -- (-1.1480148168230162,1.291200937154192) node {}
(1.569281857178206,1.291200937154192) -- (0.26822590627661547,0.013467153235845908) node {}
(2.72232299125678,0.013467153235845908) -- (1.2464683300088109,-1.0865632432073675) node {}
(2.897121701603159,-1.0865632432073675) -- (1.419605401438616,-1.3778892696171487) node {}
(2.8987833405197327,-1.3778892696171487) -- (1.421264238278007,-1.3925615307061925) node {}
(2.8987861425969155,-1.3925615307061925) -- (1.4212670366387332,-1.3931027146894364) node {}
(2.898786146313372,-1.3931027146894364) -- (1.4212670403745586,-1.393122396208641) node {};
  \end{tikzpicture}
\caption{Iterates of the approximate Douglas Rachford algorithm to find the intersection with empty interior of  ellipse and a half-plane for $\epsilon=0.245$, $\epsilon=0.120$ and exact projections.}\label{fig:ellipsehyperplane empty}
   \end{figure}

   \begin{figure}
    \begin{tikzpicture}[every mark/.append style={scale=0.6}]
     \begin{axis}[xlabel={\small Iteration (\(k\))},ylabel={\small $\|y_A^k-y_B^k\|$},width=0.3\textwidth]
       \addplot coordinates {(1,2.0663978319771825) (2,1.4212670403551897) (3,1.039379363460499) (4,0.47204033319536043) (5,2.220446049250313e-16)};
     \end{axis}
    \end{tikzpicture}
    \hfill
    \begin{tikzpicture}[every mark/.append style={scale=0.6}]
     \begin{axis}[xlabel={\small Iteration (\(k\))},ylabel={\small $\|y_A^k-y_B^k\|$},width=0.3\textwidth]
       \addplot coordinates {(1,2.0663978319771825) (2,1.2246636354109688) (3,0.8273992220396492) (4,0.5793253623854973) (5,3.1401849173675503e-16)};
     \end{axis}
    \end{tikzpicture}
    \hfill
    \begin{tikzpicture}[every mark/.append style={scale=0.6}]
     \begin{axis}[xlabel={\small Iteration (\(k\))},ylabel={\small $\|y_A^k-y_B^k\|$},width=0.3\textwidth]
       \addplot coordinates {(1,2.5777521817021856) (2,1.721077418201604) (3,1.1138318823941835) (4,0.29133076512377065) (5,0.014672261356610879) (6,0.0005411839832570809)};
     \end{axis}
    \end{tikzpicture}
      \caption{Distance $\|y_A^k-y_B^k\|$, when $\mathrm{int}(E_1\cap H) = \emptyset$ for $\epsilon=0.245$, $\epsilon=0.120$ and exact projections.} \label{fig:ellipse_hyperplane_nointerior}
   \end{figure}

%%%%%%%%%%%%%%%%%%%%%%%%%%%%%%
\section{Conclusions} \label{Sec:Conclusions}
In this paper, we proposed the approximate Douglas Rachford algorithm for finding a point in the intersection of a convex, compact set with a convex set. This method is based on the Douglas Rachford algorithm, where exact projections are replaced with inexact projections, which can be found using the Conditional Gradient algorithm. The algorithm produces a sequence of iterates that converges towards a fix point of the exact Douglas Rachford operator, and the sequence of approximate projections converges towards a point in the intersection between the two sets. Numerical results demonstrate that the proposed algorithm works well in practice, and competitively with the exact Douglas Rachford algorithm. When the interior of the intersection is non-empty both algorithms converged after finitely many iterations. In future work, we plan to extend our results to the case when neither set is bounded, and to the infinite dimensional case. We will also investigate the complexity of the proposed algorithm, and verify the finite convergence of the algorithm when the interior is nonempty.
%\bibliographystyle{habbrv}
%\bibliography{InexactDRMethod}
%\end{document}

\end{document}